\RequirePackage{fix-cm}
\documentclass[smallextended]{svjour3} 
\smartqed  
\usepackage{graphicx,color}

\usepackage{amsmath}
\usepackage{amsfonts}
\usepackage{amssymb}
\usepackage{graphicx}%

\providecommand{\U}[1]{\protect\rule{.1in}{.1in}}

\begin{document}

\title{A Remarkable Identity Involving Bessel Functions}
\author{Diego E. Dominici \and Peter M.W. Gill \and Taweetham Limpanuparb }

\institute{D. E. Dominici \at
Technische Universit\"{a}t Berlin, Stra\ss e des 17.
Juni 136, D-10623 Berlin, Germany \\
Permanent address: Department of
Mathematics, State University of New York at New Paltz, 1 Hawk Dr., New Paltz,
NY 12561-2443, USA \\
\email{dominicd@newpaltz.edu}
\and
P. M. W. Gill, T. Limpanuparb
\at
Research School of Chemistry, Australian National University, ACT 0200, Australia \\
\email{peter.gill@anu.edu.au}}

\maketitle

\begin{abstract}
We consider a new identity involving integrals and sums of Bessel functions.
The identity provides new ways to evaluate integrals of products
of two Bessel functions. 
The identity is remarkably simple and powerful since
the summand and integrand are of exactly the same form and the sum converges
to the integral relatively fast for most cases. A proof and numerical examples
of the identity are discussed.

\end{abstract}

\section{Introduction}

The Newtonian kernel
\begin{equation}
K(\boldsymbol{r},\boldsymbol{r}^{\prime})=\frac{1}{|\boldsymbol{r}%
-\boldsymbol{r}^{\prime}|},\qquad\boldsymbol{r},\boldsymbol{r}^{\prime}%
\in\mathbb{R}^{3}, \label{eq:kernel}%
\end{equation}
is ubiquitous in mathematical physics and is essential to an understanding of
both gravitation and electrostatics \cite{KelloggBook}. It is central in
classical mechanics \cite{GoldsteinBook} but plays an equally important role
in quantum mechanics \cite{LLbook} where it mediates the dominant two-particle
interaction in electronic Schr\"{o}dinger equations of atoms and molecules.

Although the Newtonian kernel has many beautiful mathematical properties, the
fact that it is both singular and long-ranged is awkward and expensive from a
computational point of view \cite{HockneyBook} and this has led to a great
deal of research into effective methods for its treatment. Of the many schemes
that have been developed, Ewald partitioning \cite{Ewald21}, multipole methods
\cite{Greengard87} and Fourier transform techniques \cite{Payne1992} are
particularly popular and have enabled the simulation of large-scale
particulate and continuous systems, even on relatively inexpensive computers.

A recent alternative \cite{Gilbert1996,ROII,ROIII,ROI} to these conventional
techniques is to resolve (\ref{eq:kernel}), a non-separable kernel, into a sum
of products of one-body functions
\begin{equation}
K(\boldsymbol{r},\boldsymbol{r}^{\prime})=\sum_{l=0}^{\infty}\sum_{m=-l}%
^{l}Y_{lm}(\boldsymbol{r})Y_{lm}(\boldsymbol{r}^{\prime})K_{l}(r,r^{\prime
})=\sum_{n,l=0}^{\infty}\sum_{m=-l}^{l}\phi_{nlm}(\boldsymbol{r})\phi
_{nlm}(\boldsymbol{r}^{\prime})\label{eq:res}%
\end{equation}
where $Y_{lm}(\boldsymbol{r})$ is a real spherical harmonic \cite[14.30.2]%
{NISTbook} of the angular part of three dimensional vector $\boldsymbol{r}$,
\begin{equation}
K_{l}(r,r^{\prime})=4\pi{\displaystyle\int\limits_{0}^{\infty}} \frac
{J_{l+1/2}(kr)J_{l+1/2}(kr^{\prime})}{k\sqrt{r r^{\prime}}}dk,
\nonumber\label{eq:tl}%
\end{equation}
$J_{l}\left( z\right)  $ is a Bessel function of the first kind \cite[10.2.2]%
{NISTbook}, and $r=|\boldsymbol{r}|$. The resolution (\ref{eq:res}) is
computationally useful because it decouples the coordinates $\boldsymbol{r}$
and $\boldsymbol{r}^{\prime}$ and allows the two-body interaction integral
\begin{equation}
E[\rho_{a},\rho_{b}]=\iint\rho_{a}(\boldsymbol{r})K(\boldsymbol{r}%
,\boldsymbol{r}^{\prime})\rho_{b}(\boldsymbol{r}^{\prime})d\boldsymbol{r}%
d\boldsymbol{r}^{\prime}, \label{eq:EJ}%
\end{equation}
between densities $\rho_{a}(\boldsymbol{r})$ and $\rho_{b}(\boldsymbol{r})$ to
be recast as
\begin{equation}
E[\rho_{a},\rho_{b}]=\sum_{n=0}^{\infty}\sum_{l=0}^{\infty}\sum_{m=-l}%
^{l}A_{nlm}B_{nlm}, \label{eq:EJ2}%
\end{equation}
where $A_{nlm}$ is a one-body integral of the product of $\rho
_{a}(\boldsymbol{r})$ and $\phi_{nlm}(\boldsymbol{r})$. If the one-body
integrals can be evaluated efficiently and the sum converges rapidly,
(\ref{eq:EJ2}) may offer a more efficient route to $E[\rho_{a},\rho_{b}]$ than
(\ref{eq:EJ}).

The key question is how best to obtain the $K_{l}$ resolution
\[
K_{l}(r,r^{\prime})=\sum_{n=0}^{\infty}K_{nl}(r)K_{nl}(r^{\prime}).
\]
Previous attempts \cite{Gilbert1996,ROII,ROI} yielded complicated $K_{nl}$
whose practical utility is questionable but, recently, we have discovered the
remarkable identity
\begin{equation}
\int_{0}^{\infty}\frac{J_{\nu}(a t)J_{\nu}(b t)}{t}dt=\sum_{n=0}^{\infty
}\varepsilon_{n}\frac{J_{\nu}(an)J_{\nu}(bn)}{n} \label{eq:identity}%
\end{equation}
where $a,b\in[0,\pi]$, $\nu=\frac{1}{2},\frac{3}{2},\frac{5}{2},\ldots$ and
$\varepsilon_{n}$ is defined by%
\begin{equation}
\varepsilon_{n}=\left\{
\begin{array}
[c]{c}%
\frac{1}{2},\quad n=0\\
1,\quad n\geq1
\end{array}
\right.  .\label{eq:epsilon}%
\end{equation}
This yields \cite{ROIV} the functions
\[
\phi_{nlm}(\boldsymbol{r})=\sqrt{\frac{4\pi\varepsilon_{n}}{rn}}%
\,J_{l+1/2}(rn)\,Y_{lm}(\boldsymbol{r}),
\]
and these provide a resolution which is valid provided $r < \pi$.

The aim of this Communication is to prove an extended version of the identity
(\ref{eq:identity}) and demonstrate its viability in approximating the
integral of Bessel functions.

\section{Preliminaries}

\vspace{0in}The Bessel function of the first kind $J_{\nu}\left(  z\right)  $
is defined by \cite[3.1 (8)]{MR1349110}%
\begin{equation}
J_{\nu}\left(  z\right)  =%
{\displaystyle\sum\limits_{n=0}^{\infty}}
\frac{\left(  -1\right)  ^{n}}{\Gamma\left(  \nu+n+1\right)  n!}\left(
\frac{z}{2}\right)  ^{\nu+2n}. \label{bessel}%
\end{equation}
It follows from (\ref{bessel}) that
\[
J_{\nu}\left(  z\right)  \left(  \frac{z}{2}\right)  ^{-\nu}%
\]
is an entire function of $z$ and we have%
\[
\underset{z\rightarrow0}{\lim}J_{\nu}\left(  z\right)  \left(  \frac{z}%
{2}\right)  ^{-\nu}=\frac{1}{\Gamma\left(  \nu+1\right)  }.
\]

Gauss' hypergeometric function is defined by \cite[2.1.2]{MR1688958}%
\begin{equation}
\ _{2}F_{1}\left(
\begin{array}
[c]{c}%
a,b\\
c
\end{array}
;z\right)  =%
{\displaystyle\sum\limits_{k=0}^{\infty}}
\frac{\left(  a\right)  _{k}\left(  b\right)  _{k}}{\left(  c\right)  _{k}%
}\frac{z^{k}}{k!}, \label{2F1}%
\end{equation}
where $\left(  u\right)  _{k}$ is the Pochhammer symbol (or rising factorial),
given by
\[
\left(  u\right)  _{k}=u\left(  u+1\right)  \cdots\left(  u+k-1\right)  .
\]
The series (\ref{2F1}) converges absolutely for $\left\vert z\right\vert <1$
\cite[2.1.1]{MR1688958}. If $\operatorname{Re}\left(  c-a-b\right)  >0,$ we
have \cite[2.2.2]{MR1688958}$\ $%
\begin{equation}
_{2}F_{1}\left(
\begin{array}
[c]{c}%
a,b\\
c
\end{array}
;1\right)  =\frac{\Gamma\left(  c\right)  \Gamma\left(  c-a-b\right)  }%
{\Gamma\left(  c-a\right)  \Gamma\left(  c-b\right)  }. \label{gauss}%
\end{equation}

Many special functions can be defined in terms of the hypergeometric function.
In particular, the Gegenbauer (or ultraspherical) polynomials $C_{n}^{\left(
\lambda\right)  }(t)$ are defined by \cite[9.8.19]{MR2656096}%
\begin{equation}
C_{n}^{\left(  \lambda\right)  }(x)=\frac{\left(  2\lambda\right)  _{n}}%
{n!}\ _{2}F_{1}\left(
\begin{array}
[c]{c}%
-n,n+2\lambda\\
\lambda+\frac{1}{2}%
\end{array}
;\frac{1-x}{2}\right)  , \label{Cn}%
\end{equation}
with $n\in\mathbb{N}_{0}$ and%
\[
\mathbb{N}_{0}=\left\{  0,1,\ldots\right\}  .
\]
In the sequel, we will use the following Lemmas.

\begin{lemma}
For $k\in\mathbb{N}_{0},$ we have
\begin{gather}
\ C_{2k}^{\left(  \mu-2k\right)  }(x)=\frac{\left(  k+1-\mu\right)  _{k}}%
{k!}\left(  1-x^{2}\right)  ^{\frac{1}{2}-\mu+2k}\label{Cn2}\\
\times\ _{2}F_{1}\left(
\begin{array}
[c]{c}%
\frac{1}{2}+k,\frac{1}{2}-\mu+k\\
\frac{1}{2}%
\end{array}
;x^{2}\right)  ,\quad\left\vert x\right\vert <1.\nonumber
\end{gather}

\end{lemma}

\begin{proof}
Using the formula \cite[3.1.12]{MR1688958}%
\begin{gather*}
\ _{2}F_{1}\left(
\begin{array}
[c]{c}%
2a,2b\\
a+b+\frac{1}{2}%
\end{array}
;\frac{x+1}{2}\right)  =\frac{\Gamma\left(  a+b+\frac{1}{2}\right)
\Gamma\left(  \frac{1}{2}\right)  }{\Gamma\left(  a+\frac{1}{2}\right)
\Gamma\left(  b+\frac{1}{2}\right)  }\ _{2}F_{1}\left(
\begin{array}
[c]{c}%
a,b\\
\frac{1}{2}%
\end{array}
;x^{2}\right) \\
-x\frac{\Gamma\left(  a+b+\frac{1}{2}\right)  \Gamma\left(  -\frac{1}%
{2}\right)  }{\Gamma\left(  a\right)  \Gamma\left(  b\right)  }\ _{2}%
F_{1}\left(
\begin{array}
[c]{c}%
a+\frac{1}{2},b+\frac{1}{2}\\
\frac{3}{2}%
\end{array}
;x^{2}\right)
\end{gather*}
in (\ref{Cn}), we obtain%
\begin{equation}
\ C_{2k}^{\left(  \mu-2k\right)  }(x)=\frac{2^{2\mu-2k-1}\Gamma\left(
\mu-2k+\frac{1}{2}\right)  \Gamma\left(  \mu-k\right)  }{\Gamma\left(
\frac{1}{2}-k\right)  \Gamma\left(  2\mu-4k\right)  \left(  2k\right)
!}\ _{2}F_{1}\left(
\begin{array}
[c]{c}%
-k,\mu-k\\
\frac{1}{2}%
\end{array}
;x^{2}\right)  , \label{Cn1}%
\end{equation}
since $\frac{1}{\Gamma\left(  -k\right)  }=0$ for $k=0,1,\ldots.$

Applying Euler's transformation \cite[2.2.7]{MR1688958}%
\[
\ _{2}F_{1}\left(
\begin{array}
[c]{c}%
a,b\\
c
\end{array}
;x\right)  =\left(  1-x\right)  ^{c-a-b}\ _{2}F_{1}\left(
\begin{array}
[c]{c}%
c-a,c-b\\
c
\end{array}
;x\right)
\]
to (\ref{Cn1}), we get
\begin{gather*}
\ C_{2k}^{\left(  \mu-2k\right)  }(x)=\frac{2^{2\mu-2k-1}\Gamma\left(
\mu-2k+\frac{1}{2}\right)  \Gamma\left(  \mu-k\right)  }{\Gamma\left(
\frac{1}{2}-k\right)  \Gamma\left(  2\mu-4k\right)  \left(  2k\right)  !}\\
\times\left(  1-x^{2}\right)  ^{\frac{1}{2}-\mu+2k}\ _{2}F_{1}\left(
\begin{array}
[c]{c}%
\frac{1}{2}+k,\frac{1}{2}-\mu+k\\
\frac{1}{2}%
\end{array}
;x^{2}\right)  ,
\end{gather*}
and (\ref{Cn2}) follows since%
\[
\frac{2^{2\mu-2k-1}\Gamma\left(  \mu-2k+\frac{1}{2}\right)  \Gamma\left(
\mu-k\right)  }{\Gamma\left(  \frac{1}{2}-k\right)  \Gamma\left(
2\mu-4k\right)  \left(  2k\right)  !}=\frac{\left(  k+1-\mu\right)  _{k}}%
{k!}.
\]

\end{proof}

\begin{lemma}
Let the function $h(x;a)$ be defined by
\[
h(x;a)=\left\{
\begin{array}
[c]{c}%
A_{k}^{\mu}\left(  a\right)  \left(  1-\frac{x^{2}}{a^{2}}\right)
^{\mu-2k-\frac{1}{2}}C_{2k}^{\left(  \mu-2k\right)  }(\frac{x}{a}),\quad0\leq
x<a\\
0,\quad a\leq x\leq\pi
\end{array}
\right.  ,
\]
where $0<a<\pi,$
\begin{equation}
A_{k}^{\mu}(a)=\frac{\left(  -1\right)  ^{k}\left(  2k\right)  !\Gamma\left(
\mu-2k\right)  2^{2\mu-2k-1}}{a^{2k+1}\Gamma\left(  2\mu-2k\right)  },
\label{A}%
\end{equation}
$\operatorname{Re}\left(  \mu\right)  >2k-\frac{1}{2}$ and $k\in\mathbb{N}%
_{0}.$

Then, $h(x;a)$ can be represented by the Fourier cosine series%
\begin{equation}
h(x;a)=%
{\displaystyle\sum\limits_{n=0}^{\infty}}
\varepsilon_{n}\frac{J_{\mu}\left(  na\right)  }{\left(  \frac{1}{2}an\right)
^{\mu}}n^{2k}\cos\left(  nx\right)  , \label{Fourierh}%
\end{equation}
where $\varepsilon_{n}$ was defined in \ref{eq:epsilon}.
\end{lemma}

\begin{proof}
For $\operatorname{Re}(\sigma)>-\frac{1}{2},\alpha>0$ and $k\in\mathbb{N}%
_{0},$ we have \cite[3.32]{MR1349110}
\begin{equation}%
{\displaystyle\int\limits_{0}^{1}}
\left(  1-t^{2}\right)  ^{\sigma-\frac{1}{2}}C_{2k}^{\left(  \sigma\right)
}(t)\cos\left(  \alpha t\right)  dt=\frac{\pi\left(  -1\right)  ^{k}%
\Gamma\left(  2k+2\sigma\right)  }{\left(  2k\right)  !\Gamma\left(
\sigma\right)  \left(  2\alpha\right)  ^{\sigma}}J_{\sigma+2k}\left(
\alpha\right)  .\label{int-geg}%
\end{equation}
Replacing $\sigma=\mu-2k$ and $\alpha=na$ in (\ref{int-geg}), we obtain%
\[
\frac{J_{\mu}\left(  na\right)  n^{2k}}{\left(  \frac{1}{2}an\right)  ^{\mu}%
}=2a%
{\displaystyle\int\limits_{0}^{1}}
A_{k}^{\mu}(a)\left(  1-t^{2}\right)  ^{\mu-2k-\frac{1}{2}}C_{2k}^{\left(
\mu-2k\right)  }(t)\cos\left(  nat\right)  dt
\]
or%
\begin{equation}
\frac{J_{\mu}\left(  na\right)  }{\left(  \frac{1}{2}an\right)  ^{\mu}}%
n^{2k}=\frac{2}{\pi}%
{\displaystyle\int\limits_{0}^{a}}
h(x;a)\cos\left(  nx\right)  dx,\label{int-rep}%
\end{equation}
and the result follows.
\end{proof}

\section{Main result}

The discontinuous integral%
\[
I\left(  a,b\right)  =%
{\displaystyle\int\limits_{0}^{\infty}}
\frac{J_{\mu}\left(  at\right)  J_{\nu}\left(  bt\right)  }{t^{\lambda}}dt,
\]
was investigated by Weber \cite{Weber}, Sonine \cite{Sonine} and Schafheitlin
\cite{Schafheitlin}. They proved that \cite[13.4 (2)]{MR1349110}%
\begin{equation}
I\left(  a,b\right)  =\frac{a^{\lambda-\nu-1}b^{\nu}\Gamma\left(  \frac
{\nu+\mu-\lambda+1}{2}\right)  }{2^{\lambda}\Gamma\left(  \nu+1\right)
\Gamma\left(  \frac{\lambda+\mu-\nu+1}{2}\right)  }\ _{2}F_{1}\left(
\begin{array}
[c]{c}%
\frac{\nu+\mu-\lambda+1}{2},\frac{\nu-\mu-\lambda+1}{2}\\
\nu+1
\end{array}
;\left(  \frac{b}{a}\right)  ^{2}\right)  ,\label{Sonine}%
\end{equation}
for
\begin{equation}
\operatorname{Re}(\mu+\nu+1)>\operatorname{Re}\left(  \lambda\right)
>-1\label{conditions}%
\end{equation}
and $0<b<a.$ The corresponding expression for the case when $0<a<b$, is
obtained from (\ref{Sonine}) by interchanging $a,b$ and also $\mu,\nu$. When
$a=b,$ we have \cite[13.41 (2)]{MR1349110}%
\[
I\left(  a,a\right)  =\frac{a^{\lambda-1}\Gamma\left(  \dfrac{\nu+\mu
-\lambda+1}{2}\right)  \Gamma\left(  \lambda\right)  }{2^{\lambda}%
\Gamma\left(  \dfrac{\lambda+\mu-\nu+1}{2}\right)  \Gamma\left(
\dfrac{\lambda+\nu-\mu+1}{2}\right)  \Gamma\left(  \dfrac{\nu+\mu+\lambda
+1}{2}\right)  },
\]
provided that $\operatorname{Re}(\mu+\nu+1)>\operatorname{Re}\left(
\lambda\right)  >0.$ This result also follows from Gauss' summation formula
(\ref{gauss}) and (\ref{Sonine}).

\begin{theorem}
\label{Theorem1}\vspace{0in}If $0<b<a<\pi,\ \operatorname{Re}\left(
\mu\right)  >2k-\frac{1}{2},$ $\operatorname{Re}\left(  \nu\right)  >-\frac
{1}{2},$ $k\in\mathbb{N}_{0},$ and%
\[
S_{k}\left(  a,b\right)  =%
{\displaystyle\sum\limits_{n=0}^{\infty}}
\varepsilon_{n}\frac{J_{\mu}\left(  an\right)  }{\left(  \frac{1}{2}an\right)
^{\mu}}\frac{J_{\nu}\left(  bn\right)  }{\left(  \frac{1}{2}bn\right)  ^{\nu}%
}\left(  \frac{an}{2}\right)  ^{2k},
\]
then,%
\begin{equation}
S_{k}\left(  a,b\right)  =\frac{\Gamma\left(  k+\frac{1}{2}\right)  }%
{a\Gamma\left(  \nu+1\right)  \Gamma\left(  \mu-k+\frac{1}{2}\right)  }%
\ _{2}F_{1}\left(
\begin{array}
[c]{c}%
\frac{1}{2}+k,\frac{1}{2}-\mu+k\\
\nu+1
\end{array}
;\left(  \frac{b}{a}\right)  ^{2}\right)  .\label{Sum}%
\end{equation}

\end{theorem}

\begin{proof}
Multiplying (\ref{Fourierh}) by $A_{0}^{\nu}\left(  b\right)  \frac{2}{\pi
}\left(  1-\frac{x^{2}}{b^{2}}\right)  ^{\nu-\frac{1}{2}}$ and integrating
from $0$ to $b,$ we get%
\begin{equation}
S_{k}\left(  a,b\right)  =\frac{2}{\pi}A_{0}^{\nu}\left(  b\right)  A_{k}%
^{\mu}\left(  a\right)
{\displaystyle\int\limits_{0}^{b}}
\left(  1-\frac{x^{2}}{b^{2}}\right)  ^{\nu-\frac{1}{2}}\left(  1-\frac{x^{2}%
}{a^{2}}\right)  ^{\mu-2k-\frac{1}{2}}C_{2k}^{\left(  \mu-2k\right)  }%
(\frac{x}{a})dx, \label{integral}%
\end{equation}
where we have used the integral representation (\ref{int-rep}). Setting $x=bt$
and $b=\omega a$ in (\ref{integral}), we obtain%
\begin{equation}
S_{k}\left(  a,b\right)  =\frac{2b}{\pi}A_{0}^{\nu}\left(  b\right)
A_{k}^{\mu}\left(  a\right)
{\displaystyle\int\limits_{0}^{1}}
\left(  1-t^{2}\right)  ^{\nu-\frac{1}{2}}\left(  1-\omega^{2}t^{2}\right)
^{\mu-2k-\frac{1}{2}}C_{2k}^{\left(  \mu-2k\right)  }(\omega t)dt.
\label{integral1}%
\end{equation}
Thus, we can re-write (\ref{integral1}) as%
\begin{align*}
S_{k}\left(  a,b\right)   &  =\frac{2b}{\pi}A_{0}^{\nu}\left(  b\right)
A_{k}^{\mu}\left(  a\right)  \frac{2^{2\mu-2k-1}\Gamma\left(  \mu-2k+\frac
{1}{2}\right)  \Gamma\left(  \mu-k\right)  }{\Gamma\left(  \frac{1}%
{2}-k\right)  \Gamma\left(  2\mu-4k\right)  \left(  2k\right)  !}\\
&  \times%
{\displaystyle\int\limits_{0}^{1}}
\left(  1-t^{2}\right)  ^{\nu-\frac{1}{2}}\ _{2}F_{1}\left(
\begin{array}
[c]{c}%
\frac{1}{2}+k,\frac{1}{2}-\mu+k\\
\frac{1}{2}%
\end{array}
;\omega^{2}t^{2}\right)  dt,
\end{align*}
or, using (\ref{A}) and changing variables in the integral,%
\begin{align*}
S_{k}\left(  a,b\right)   &  =\frac{\left(  -1\right)  ^{k}2^{2k}\sqrt{\pi}%
}{a^{2k+1}\Gamma\left(  \mu-k+\frac{1}{2}\right)  \Gamma\left(  \nu+\frac
{1}{2}\right)  \Gamma\left(  \frac{1}{2}-k\right)  }\\
&  \times%
{\displaystyle\int\limits_{0}^{1}}
s^{-\frac{1}{2}}\left(  1-s\right)  ^{\nu-\frac{1}{2}}\ _{2}F_{1}\left(
\begin{array}
[c]{c}%
\frac{1}{2}+k,\frac{1}{2}-\mu+k\\
\frac{1}{2}%
\end{array}
;\omega^{2}s\right)  ds.
\end{align*}

Recalling the formula \cite[Th. 2.2.4]{MR1688958}%
\[
_{2}F_{1}\left(
\begin{array}
[c]{c}%
a,b\\
c
\end{array}
;x\right)  =\frac{\Gamma\left(  c\right)  }{\Gamma\left(  d\right)
\Gamma\left(  c-d\right)  }%
{\displaystyle\int\limits_{0}^{1}}
t^{d-1}\left(  1-t\right)  ^{c-d-1}\ _{2}F_{1}\left(
\begin{array}
[c]{c}%
a,b\\
d
\end{array}
;xt\right)  dt,
\]
valid for $\operatorname{Re}\left(  c\right)  >\operatorname{Re}(d)>0,$
$x\in\mathbb{C}\setminus\left[  1,\infty\right)  ,$ we conclude that%
\begin{align*}
S_{k}\left(  a,b\right)   &  =\frac{\left(  -1\right)  ^{k}2^{2k}\sqrt{\pi}%
}{a^{2k+1}\Gamma\left(  \mu-k+\frac{1}{2}\right)  \Gamma\left(  \nu+\frac
{1}{2}\right)  \Gamma\left(  \frac{1}{2}-k\right)  }\\
&  \times\frac{\sqrt{\pi}\Gamma\left(  \nu+\frac{1}{2}\right)  }{\Gamma\left(
\nu+1\right)  }\ _{2}F_{1}\left(
\begin{array}
[c]{c}%
\frac{1}{2}+k,\frac{1}{2}-\mu+k\\
\nu+1
\end{array}
;\omega^{2}\right)  .
\end{align*}
But since \cite[1.2.1]{MR1688958}%
\[
\Gamma\left(  k+\frac{1}{2}\right)  \Gamma\left(  \frac{1}{2}-k\right)
=\left(  -1\right)  ^{k}\pi,\quad k=0,1,\ldots,
\]
the result follows.
\end{proof}

The special case of Theorem \ref{Theorem1} in which $k=0,$ was derived by
Cooke in \cite{Cooke}, as part of his work on Schl\"{o}milch series.

\begin{corollary}
\label{Corollary 1}If $0<b<a<\pi,\ \operatorname{Re}\left(  \mu\right)
>2k-\frac{1}{2},\operatorname{Re}\left(  \nu\right)  >-\frac{1}{2}%
,k\in\mathbb{N}_{0},$ then%
\[%
{\displaystyle\int\limits_{0}^{\infty}}
\frac{J_{\mu}\left(  at\right)  J_{\nu}\left(  bt\right)  }{t^{\mu+\nu-2k}}dt=%
{\displaystyle\sum\limits_{n=0}^{\infty}}
\varepsilon_{n}\frac{J_{\mu}\left(  an\right)  J_{\nu}\left(  bn\right)
}{n^{\mu+\nu-2k}}.
\]

\end{corollary}

\begin{proof}
The result follows immediately from (\ref{Sonine}) and (\ref{Sum}), after
taking $\lambda=\mu+\nu-2k.$ Note that since for all $k\in\mathbb{N}_{0}$
\[
\operatorname{Re}(\mu+\nu+1)=\operatorname{Re}\left(  2k+1+\lambda\right)
>\operatorname{Re}\left(  1+\lambda\right)  >\operatorname{Re}\left(
\lambda\right)
\]
and%
\[
\operatorname{Re}\left(  \lambda\right)  =\operatorname{Re}\left(  \mu
+\nu-2k\right)  >-1,
\]
the conditions (\ref{conditions}) are satisfied.
\end{proof}

\begin{corollary}
If $0<a,b<\pi,$ and $\operatorname{Re}(\mu)>0,$ then%
\[%
{\displaystyle\sum\limits_{n=0}^{\infty}}
\varepsilon_{n}\frac{J_{\mu}\left(  an\right)  J_{\mu}\left(  bn\right)  }{n}=%
{\displaystyle\int\limits_{0}^{\infty}}
\frac{J_{\mu}\left(  at\right)  J_{\mu}\left(  bt\right)  }{t}dt=\left\{
\begin{array}
[c]{c}%
\frac{1}{2\mu}\left(  \frac{a}{b}\right)  ^{\mu},\quad a\leq b\\
\frac{1}{2\mu}\left(  \frac{b}{a}\right)  ^{\mu},\quad a\geq b
\end{array}
\right.  .
\]

\end{corollary}

\begin{proof}
The result is a consequence of Corollary \ref{Corollary 1} and a special case
of the integral (\ref{Sonine}) (see \cite[13.42 (1)]{MR1349110}).
\end{proof}

\section{Numerical results}

\begin{figure}[ptb]
\caption{$\frac{1}{5}\left[ \frac{\min(a,b)}{\max(a,b)}\right]^{5/2}$ (solid) and $T_{10}(a,b)$ (dashed), for $a=\frac{\pi}{4}, \frac{2\pi}{4}, \frac{3\pi}{4}$ and $b\in[0,2\pi]$ }
\centering
\includegraphics[width=.77\textwidth]{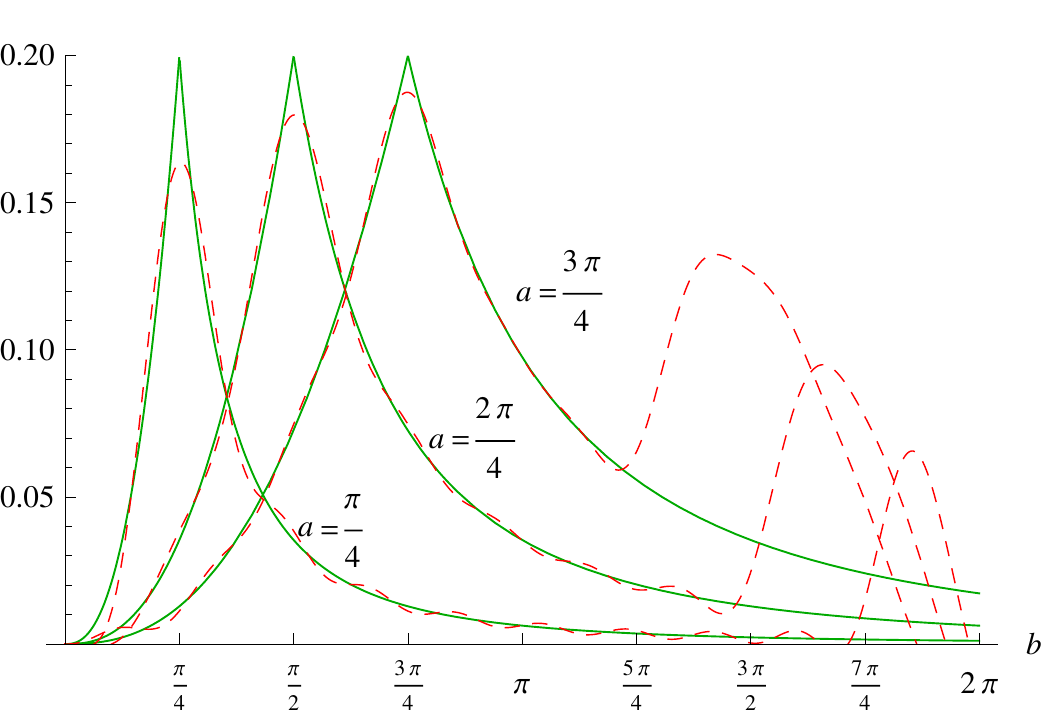}\end{figure}

\begin{figure}[ptb]
\caption{Truncation error $R_{20}(a,b)$ for $a,b\in[-2\pi,2\pi]$}%
\centering
\includegraphics[width=.77\textwidth]{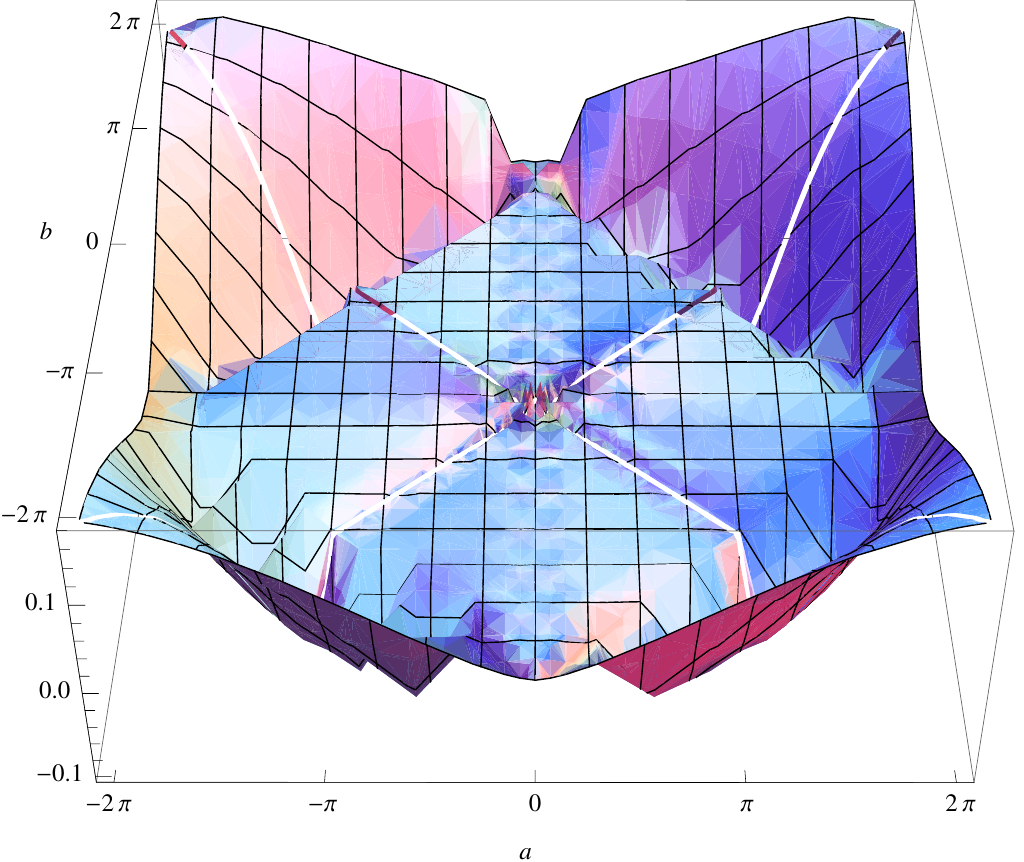}\end{figure}

\begin{table}[ptb]
\caption{$R_{N}(a,b)$ for $a = \pi/2$ and various $b$ and $N$}%
\begin{tabular}
[c]{c|rrrrrrr}\hline \hline
& \multicolumn{6}{c}{$b$} & \\
$N$ & $\pi/4$ & 2$\pi/4$ & 3$\pi/4$ & 4$\pi/4$ & 5$\pi/4$ & 6$\pi/4$ &
\\\hline
1 & 3.15E-02 & 1.81E-01 & 3.15E-02 & -2.37E-02 & -4.12E-02 & -3.09E-2 & \\
10 & -3.10E-03 & 2.02E-02 & -2.40E-03 & 1.45E-04 & 1.74E-03 & -1.16E-2 & \\
100 & -1.79E-06 & 2.03E-03 & -4.81E-07 & -1.39E-07 & 7.46E-07 & -1.17E-3 & \\
1000 & 1.81E-09 & 2.03E-04 & 4.86E-10 & -1.37E-10 & -7.46E-10 & -1.17E-4 & \\
10000 & 1.81E-12 & 2.03E-05 & 4.86E-13 & -1.37E-13 & -7.46E-13 & -1.17E-5 & \\
\hline \hline
\end{tabular}
\end{table}

The efficient and accurate computation of the $_2F_1$ function in (\ref{Sonine}) can be numerically challenging \cite{NR2007}, and we note that Corollary 1 provides a new route to its evaluation.  Because the numerator is bounded $|J_{\mu}\left(  at\right)  J_{\mu}\left(  bt\right)| \le 1$ \cite[13.42 (10)]{MR1349110}, the Weierstrass M-test shows that the series converges uniformly as long as $\mu+\nu-2k>1$ and $\mu,\nu>0$.

To explore the rate of convergence of the sum in Corollary 2, we choose $\nu= 5/2$.  The exact value of the integral is
\begin{equation}
	\int\limits_{0}^{\infty} \frac{J_{5/2}(at)  J_{5/2}(bt) }{t}dt = \frac{1}{5} \sqrt{\frac{|a|}{a}} \sqrt{\frac{|b|}{b}}
		\left[ \frac{\min(|a|,|b|)}{\max(|a|,|b|)} \right]^{5/2} 
\end{equation}
and truncation of the infinite series yields the finite sum
\begin{equation}
	T_N(a,b)=\sum_{n=1}^{N}\frac{J_{5/2}\left(  an\right)J_{5/2}\left(  bn\right)  }{n}
\end{equation}
and a truncation error
\begin{equation}
	R_N(a,b)=\frac{1}{5}\left[ \frac{\min(|a|,|b|)}{\max(|a|,|b|)} \right]^{5/2} - \sqrt\frac{|a|}{a} \sqrt\frac{|b|}{b} \ T_{N}(a,b). 
\end{equation}
The integral in (22) and sum in (23) are illustrated in Figure 1 and their difference (24) is shown in Figure 2 and Table 1.

The truncation errors in Table 1 reveal that the rate of convergence of the
Bessel sum is strongly dependent on the value of $a$ and $b$. For $b = 2\pi/4$ and $b
= 6\pi/4$, the truncation error appears to decay as $1/N$ but, for the other
values of $b$, it appears to decay as $1/N^{3}$. Although these empirical
results are interesting, we have not been able to determine analytically the
decay behavior as a function of $a$ and $b$.

The excellent agreement in Figure 1 and apparently flat plateau in Figure 2 strongly suggest that
Corollary 2 is true over the larger domain $|a| + |b| < 2\pi$. However, we
have not yet managed to find a proof for this.

\section{Concluding remark}

We provide a rigorous proof that the identity (5) is valid on the square domain
$a,b \in(0,\pi)$. A numerical study indicates that the rate of convergence of the sum in the identity is sensitive to
the values of $a$ and $b$ and further work to quantify this would be helpful.
Generalization of the identity is possible and should be explored in
the future work.

\section*{Acknowledgement}

The work of D. Dominici was supported by a Humboldt Research Fellowship for
Experienced Researchers from the Alexander von Humboldt Foundation.  P.M.W. Gill thanks
the Australian Research Council (Grants DP0984806 and DP1094170) for funding
and the NCI National Facility for a generous grant of supercomputer time. 
T. Limpanuparb thanks the Development and Promotion of Science and Technology
Talents Project for a Royal Thai Government PhD scholarship.  
\bibliographystyle{abbrv}

\end{document}